\documentclass[12pt]{amsart}

\usepackage{amsmath, amsthm, amssymb, amsfonts, amscd, color, enumerate, graphicx}

\def\CC{{\mathbb C}}
\def\NN{{\mathbb N}}
\def\QQ{{\mathbb Q}}
\def\RR{{\mathbb R}}
\def\TT{{\mathbb T}}
\def\ZZ{{\mathbb Z}}

\def\Xcal{{\mathcal X}}

\def\sup{\mathrm{sup}}
\def\max{\mathrm{max}}

\newcommand{\ndiv}{\hspace{-4pt}\not|\hspace{2pt}} 

\theoremstyle{plain}
\newtheorem{thm}{Theorem}

\newtheorem{cor}[thm]{Corollary}
\newtheorem{prop}[thm]{Proposition}
\newtheorem{lem}[thm]{Lemma}

\theoremstyle{definition}
\newtheorem{dfn}{Definition}

\newtheorem{rem}{Remark}

\title[A $p$-adic LeVeque-type Inequality]{A LeVeque-type Inequality on the ring of $p$-adic integers}

\author{Naveen Somasunderam}

\address{Naveen Somasunderam; Department of Mathematics; Oregon State University; Corvallis OR 97331 U.S.A.}

\email{somasunn@math.oregonstate.edu}

\date{April 02, 2019.}

\begin{document}

\maketitle
\begin{abstract}
We derive an inequality on the discrepancy of sequences on the ring of $p$-adic integers $\ZZ_p$ using techniques from Fourier analysis. The inequality is used to obtain an upper bound on the discrepancy of the sequence $\alpha_n = na +b$, where $a$ and $b$ are elements of $\ZZ_p$. This is a $p$-adic analogue of the classical LeVeque inequality on the circle group $\RR/\ZZ$.   
\end{abstract}

\section{Introduction}
\label{section: introduction}
The theory of  equidistribution of sequences modulo one was initiated by Hermann Weyl in $1916$. Since then, it has spurred a lot of interest in many areas of mathematics, including number theory, harmonic analysis, and ergodic theory. The standard reference in this subject is Kuipers and Niederreiter \cite{book:MR0419394}. 

Equidistribution of sequences on the ring of $p$-adic integers was previously studied in \cite{article:MR0237441, article:MR0245528, article:MR0147467}. In particular, Cugiani in \cite{article:MR0147467} defines equidistribution and shows that the sequence $na + b$ is equidistributed if $a$ is a unit. Beer does a quantitative analysis in \cite{article:MR0237441} and \cite{article:MR0245528}. 
Our aim is to derive a LeVeque-type inequality on the discrepancy of a finite sequence using Fourier analysis.

Let $| \cdot |_p$ denote the $p$-adic absolute value on $\QQ_p$, and let
\[ \ZZ_p = \{ x \,| \,\,  |x|_p \leq 1    \},
\]  
be the ring of $p$-adic integers. Any element of $\ZZ_p$ can be given a unique canonical expansion of the form $x = a_0 + a_1p + a_2p^2 + ....$, where the $a_i$ are elements of $\{0,1,2,...,p-1\}$ (see for example \cite{book:MR1488696,book:MR2298943}). 

For $k \geq 0$, and $a \in \ZZ_p$,  we denote by 
\begin{eqnarray*}
D(a,1/p^k) &=& \{ x \,| \,\,  |x-a|_p \leq 1/p^k\} \cr
          &=&  a + p^k\ZZ_p,
\end{eqnarray*}
 a disc of radius $1/p^k$ centered at $a$. Note that $\ZZ_p$ can be written as the union of $p^k$ disjoint discs of the form
\[    
     \ZZ_p = \bigcup_{j=0}^{p^k-1} D(j,1/p^k). 
\]
Hence, it is natural to define a notion of equidistribution using such sets. 
\begin{dfn}
\label{dfn: equidist}
A sequence $\{ \alpha_n \}$ is said to be equidistributed in $\ZZ_p$ if for every $a$ in $\ZZ_p$ and every $k \in \NN $, we have 
\[ \lim_{N \to \infty} \left|  \frac{\left|D(a,1/p^k) \cap \{\alpha_1,...,\alpha_N \} \right|}{N}  -  \frac{1}{p^k} \right| = 0.
\]
\end{dfn}
That is, the proportion of the first $N$ elements of $\{ \alpha_n\}$ lying in a disc $D(a,1/p^k)$ is equal to its measure in the limit of large $N$, and this holds true for all such discs.

This definition of equidistribution in $\ZZ_p$ was first given by Cugiani in \cite{article:MR0147467}, where propositions \ref{prop: Weyl} and \ref{prop: linear-sequence} were also proved. The details are also given in Kuipers and Niederreiter \cite{book:MR0419394}. One also wants to measure how well a sequence distributes itself. To this end, we define the notion of discrepancy to quantify the idea that some sequences are better equidistributed than others.
\begin{dfn}
\label{dfn: discrepancy}
The discrepancy of a finite sequence $\{\alpha_1,\alpha_2,...,\alpha_N \}$ in $\ZZ_p$ is 
\[  D_N  = \sup_{a \in \ZZ_p, \, k \in \NN} \left| \frac{\left|D(a,1/p^k) \cap \{\alpha_1,...,\alpha_N \} \right|}{N} - \frac{1}{p^k} \right|. 
\]
\end{dfn}
Some elementary arguments show that 
\[   \frac{1}{N}  \leq D_N \leq 1.
\]
The main aim of this paper is to prove a Fourier analytic upper bound on the discrepancy of a set of $N$ elements $\{\alpha_1,\alpha_2,...,\alpha_N \}$  in $\ZZ_p$. 
 
 Let $\ZZ(p^\infty)$ denote the Pr\"{u}fer $p$-group, the group of all $p$-th power roots of unity in $\CC$. Suppose that $\zeta \in \ZZ(p^\infty)$ has order $p^n$, and let $x \in \ZZ_p$ have the canonical expansion $x = a_0 + a_1p + a_2p^2 +.... + a_{n}p^{n} + .....$. Then we interpret the notation $\zeta^x$ as
\[     \zeta^x =  \zeta^{a_0 + a_1p + a_2p^2 +.... + a_{n-1}p^{n-1}}.
\]
Every element of $\ZZ(p^\infty)$ has finite order, and we denote the order of $\zeta \in \ZZ(p^\infty)$ by $\|\zeta\|$. 
\begin{thm}[Main Theorem] 
\label{thm: Main}
The discrepancy of a finite sequence $\{\alpha_1,...,\alpha_N\}$ in $\ZZ_p$ is bounded by 
\begin{eqnarray*}
D_N &\leq & C(p) 
    \left(\sum_{\zeta \in \ZZ(p^\infty)\backslash \{1\}  }   \frac{1}{\|\zeta \|^3} \left|\frac{1}{N} \sum_{n=1}^{N} \zeta^{\alpha_n} \right|^2 \right)^{\frac{1}{4}},
\end{eqnarray*}
where $C(p)$ is a constant dependent on $p$.
\end{thm}
As an example application of Theorem \ref{thm: Main}, we have the following corollary
\begin{cor}
\label{cor: Linear-sequence}
The sequence $na + b$ where $a$ is a unit in  $\ZZ_p$ has discrepancy
\[   D_N   = O\left( N^{-1/2} \right). 
\]
\end{cor}

Some quantitative results on the discrepancy of $p$-adic sequences were done by Beer in \cite{article:MR0237441} and \cite{article:MR0245528}. In particular, the author proves in \cite{article:MR0237441} that the discrepancy of the sequence $na + b$ with $a$ a unit is exactly equal to $D_N = N^{-1}$, the best possible. 

It is not surprising that the LeVeque type inequality gives us a weaker bound, as this is the case in the classical setting on $\RR/\ZZ$. Montogomery in \cite{book:MR1297543} provides a detailed discussion and considers some examples. In particular, the sequence $n\theta$ where $\theta = \frac{1 +\sqrt{5} }{ 2}$ has discrepancy $D_N \ll \log(N)/N$, where as the use of the LeVeque inequality gives only $D_N \ll N^{-2/3}$. 

Our paper is structured as follows. In section \ref{section: fanalysis-padic} we set up the relevant Fourier analysis that is required for our calculations. We prove the main theorem in section \ref{section: Main-theorem}. We analyze the quantitative behavior of the linear sequence $\alpha_n= na + b$ in section \ref{section: Linear-sequence}, and prove Corollary \ref{cor: Linear-sequence}. 


\section{Fourier analysis on $\ZZ_p$}
\label{section: fanalysis-padic} 
If $G$ is a compact abelian group, then the set of all continuous group homomorphisms (or characters) from $G$ to the multiplicative unit circle $\TT = \{z \in \CC \, | \, |z| = 1 \}$ forms a discrete group under multiplication, the Pontyagrin dual group $\widehat{G}$ (see for example \cite{book:MR1038803}). Note that $\ZZ_p$ is a compact abelian group. The next lemma states that the dual group $\widehat{\ZZ}_p$ of $\ZZ_p$ is isomorphic to the Pr\"{u}fer $p$-group $\ZZ(p^\infty)$. The result is known, but we include a proof due to the lack of a suitable reference.  
\begin{lem}
	\label{lem: dualgroup} 
	For each $\zeta \in \ZZ(p^\infty)$, the map $x \mapsto  \zeta^x $ is a character of $\ZZ_p$. Moreover, the map
	\begin{eqnarray*}
		\Psi: \,\, \ZZ(p^\infty) &\longrightarrow  &  \widehat{\ZZ}_p \cr
		\zeta             &\longmapsto &   (x \mapsto \zeta^x )
	\end{eqnarray*}
	is an isomorphism from the  Pr\"{u}fer $p$-group $\ZZ(p^\infty)$ to the Pontyagrin dual group of $\ZZ_p$.
\end{lem}

\begin{proof}
	It is easily shown that the map $x \mapsto \zeta^x$ is a character of $\ZZ_p$.  To show the injectivity of $\Psi$, suppose that $\zeta_1^x = \zeta_2^x$ for all $x$ in $\ZZ_p$. Then picking $x=1$, we get $\zeta_1 = \zeta_2$. 
	
	We need argue that $\Psi$ is surjective. Let $\gamma$ be in the dual group of $\ZZ_p$. Since $\gamma(0) = 1$ and $\gamma$ is continuous, there exists a disc of radius $1/p^n$ centered at zero $D = p^n\ZZ_p$, such that $|\gamma(x) - \gamma(0)|< 1$ for all $x$ in $D$, and we can pick a smallest $n$ such that this is true. Moreover, since $D$ is a subgroup of $\ZZ_p$ we must have that the image $\gamma(D)$ is a subgroup of $\TT$. 
	
	Note that there does not exist any non-trivial subgroup of $\TT$ satisfying the condition $|x -y| < 1$ for all elements $x$ and $y$ in the subgroup.  Hence, we conclude that $\gamma(D) = \{ 1 \}$. 
	
	Now suppose that $\gamma(1) = \zeta = e^{2\pi i \theta}$ for some $\theta$ in $[0, 1)$. Then, $\gamma(p^n) = \gamma(1)^{p^n} = e^{2\pi i p^n \theta} = 1$ or $p^n \theta \in \ZZ$. We conclude that $\theta = \frac{m}{p^n}$, where $p \ndiv m$ by the minimality of $n$. 
	
	For any integer value $k$, we have $\gamma(k) = \gamma(1)^k = \zeta^k$. This completely determines $\gamma$, since we can write $\ZZ_p$ as the union of $p^n$ disjoint balls $\ZZ_p = \cup_{k=0}^{p^n-1} \left( k + p^n\ZZ_p \right)$ and for any $x \in \ZZ_p$ we have $x = k + p^ny$, and $\gamma(x) = \gamma(k)$. We conclude that 
	\[
	\gamma(x) = \zeta^x,
	\]
	for all $x$ in $\ZZ_p$ where $\zeta = e^{\frac{2\pi i m }{p^n}}$.
\end{proof}

Using Lemma ~\ref{lem: dualgroup}, we shall express the Fourier series of any $f \in L^1(\ZZ_p)$ in terms of the elements of $\ZZ(p^\infty)$. As a compact group there exists a normalized Haar measure $\mu$ on $\ZZ_p$ (see for example \cite{book:MR1681462}). Let $f \in L^1(\ZZ_p)$. The  Fourier coefficients of $f$ are given by 
\begin{eqnarray*}
	\hat{f}(\zeta) = \int_{\ZZ^p} f(x) \zeta^{-x} d\mu,      
\end{eqnarray*}
and the Fourier inversion formula gives
\[
f(x) = \sum_{\zeta \in  \ZZ(p^\infty)} \hat{f}(\zeta) \zeta^x,
\] 
whenever $\sum_{\zeta \in  \ZZ(p^\infty)} |\hat{f}(\zeta)| < \infty$. 

A Weyl type criterion holds for equidistribution in $\ZZ_p$. We state it here in terms of the elements of $\ZZ(p^\infty)$, although it holds for a more general class of Riemann integrable functions on $\ZZ_p$ (see \cite{book:MR0419394}).  
\begin{prop}[Weyl's Criterion]
	\label{prop: Weyl}
	A sequence $\{ \alpha_n \}$ is equidistributed in $\ZZ_p$ if and only if for every non-trivial $\zeta$ in $ \ZZ(p^\infty) $ we have
	\[
	\lim_{N \to \infty} \frac{1}{N} \sum_{n=1}^N \zeta^{\alpha_n} = 0.
	\]
\end{prop}

We denote by $\Xcal_{D(a,R)}(x)$ the characteristic function of the disc $D(a,R)$ centered at $a$ of radius $R$. We have the following change of variables formula, the proof of which is elementary and we omit. 

\begin{prop}
	\label{prop: subformula}
	Let $f: \ZZ_p \to \CC$ be an integrable function. Then 
	\[  \int_{\ZZ_p} \Xcal_{D(a,1/p^k)}(x)f(x) \, d\mu(x) = \frac{1}{p^k} \int_{\ZZ_p} f(a + p^kx) \, d\mu(x). 
	\]
\end{prop} 
We use Proposition \ref{prop: subformula} to calculate the Fourier coefficients of the characteristic function of a disc.

\begin{lem}
	\label{lemma: charfun}
	The Fourier coefficients of the characteristic function  $\Xcal_{D(a,1/p^k)}(x)$ are
	\[
	\widehat{\Xcal}_{D(a,1/p^k)}(\zeta) = \left\{ 
	\begin{array}{cc}
	\zeta^{-a} p^{-k} &  \text{if \,\,}  \|\zeta\| \leq p^k, \\
	& \\
	0  &    \text{if \,\,}  \|\zeta\| > p^k. 
	\end{array} 
	\right. 
	\]
\end{lem} 

\begin{proof}[Proof of Lemma ~\ref{lemma: charfun}]
	Suppose that $ \|\zeta\| \leq p^k $, then $\zeta^{p^kx} = 1$ for all $x$ in $\ZZ_p$. Therefore, we have 
	\begin{eqnarray*}
		\int_{\ZZ_p} \Xcal_{D(a,1/p^k)}(x) \zeta^{-x} \, d\mu(x) &=& p^{-k} \int_{\ZZ_p} \zeta^{-(a + p^kx)} \, d\mu(x) \cr
		&=&  \zeta^{-a} p^{-k} \int_{\ZZ_p} \zeta^{-p^kx} \, d\mu(x) \cr
		&=& \zeta^{-a} p^{-k}.
	\end{eqnarray*}  
	On the other hand suppose $ \|\zeta\| > p^k$, and let $\omega = \zeta^{p^k}$. Then $\|\omega\| = \|\zeta\| / p^k > 1$ and hence
	\begin{eqnarray*}
		\int_{\ZZ_p} \Xcal_{D(a,1/p^{k})}(x) \zeta^{-x} \,  d\mu(x) &=& \zeta^{-a}p^{-k} \int_{\ZZ_p} \zeta^{-p^kx} \, d\mu(x) \cr 
		&=& \zeta^{-a} p^{-k} \int_{\ZZ_p} \omega^{-x} \, d\mu(x) \cr
		&=& 0.
	\end{eqnarray*}
\end{proof}


\section{Proof of the main theorem}
\label{section:  Main-theorem}

Let $\{\alpha_1,\alpha_2,...,\alpha_N \}$ be a finite sequence in $\ZZ_p$. Define the function $f: \ZZ_p \times \ZZ_p \longrightarrow \RR $ 
\[
f(x,y) = \frac{ \left| \{\alpha_1, \alpha_2,...,\alpha_N \} \, \cap D(x,|y|_p) \right| }{N} - |y|_p, 
\] 
where $D(x,|y|_p)$ is a disc of radius $|y|_p$ centered at $x$. The discrepancy of the points $\{\alpha_1,...,\alpha_N \}$ is then
\[
D_N = \sup_{x, y \in \ZZ_p} \, \left| f(x,y) \right|.
\] 
We suppress the $p$ in $|\cdot|_p$ as it would be clear from the context. 
We can also write 
\begin{eqnarray*}
	f(x,y) &=& \frac{1}{N} \left( \sum_{n=1}^{N} \Xcal_{D(x,|y|)} (\alpha_n) \right) - |y| \cr
	& & \cr
	&=& \frac{1}{N} \left( \sum_{n=1}^{N} \Xcal_{D(\alpha_n, |y|)} (x) \right) - |y|.
\end{eqnarray*} 

Our proof of  Theorem \ref{thm: Main} proceeds as follows. We shall bound the $L^2$ norm $\displaystyle \|f\|_2^2 = \iint_{\ZZ_p^2} |f(x,y)|^2 \, d\mu(x) d\mu(y)$ from below by $D_N^4$ using geometrical arguments, and from above by using Parseval's theorem. The two steps are given below as lemmas

\begin{lem} 
	\label{lemma: bound1}
	The discrepancy $D_N$ is bounded by
	\[  D_N^4 \leq C_1(p) \|f\|_2^2,
	\]
	where $C_1(p)$ is a constant dependent on $p$.
\end{lem}

\begin{lem} 
	\label{lemma: bound2}
	The $L^2$ norm of the function $f$ is bounded by
	\[  \|f\|_2^2 \leq C_2(p) \sum_{\zeta \in \ZZ(p^\infty)\backslash \{1 \} }   \frac{1}{\|\zeta\|^3} \left|\frac{1}{N} \sum_{n=1}^{N} \zeta^{x_n} \right|^2.
	\]
	where $C_2(p)$ is a constant dependent on $p$.
\end{lem}

The proof of Theorem \ref{thm: Main} then follows by combining Lemmas \ref{lemma: bound1} and \ref{lemma: bound2}. 

\begin{rem}
	For $x >0$, we use the notation $\lfloor x \rfloor $ and $\lceil x \rceil$ to denote 
	\[
	\lfloor x \rfloor = \max\{p^k \, | \, k \in \ZZ,\, p^k \leq x  \}
	\]
	\[
	\lceil x \rceil  = \min\{p^k \, | \, k \in \ZZ,\, x \leq p^k  \}.
	\] \\
	Note that $ \lfloor x \rfloor \leq x < p \lfloor x \rfloor $ and  $\frac{1}{p}\lceil x \rceil < x \leq  \lceil x \rceil   $.
\end{rem}

\begin{proof}[Proof of Lemma ~\ref{lemma: bound1}] 
	Pick a point $(x_0,y_0)$ for which $f(x_0,y_0)$ is not zero. 
	We consider each of the two possibilities $f(x_0,y_0) > 0$ and $f(x_0,y_0) <0$ separately. Our strategy in each case is to find a small neighborhood around the point $(x_0,y_0)$ where $|f(x,y)|$ is bounded away from zero. Using this fact and integrating over this neighborhood, we produce a bound of the form $\|f\|_2^2 \geq C(p) |f(x_0,y_0)|^4$, where $C(p)$ is a constant depending only on $p$. 
	\subsection*{Case 1} Suppose that $\Delta = f(x_0,y_0) > 0.$
	Let $R = \lfloor \Delta +|y_0| \rfloor$. Since, $|y_0| < |y_0| + \Delta$ and $|y_0|$ is in the value group of $\QQ_p$, we have $|y_0| \leq R$. We consider the two cases $|y_0| < R$ and $|y_0| = R$.
	\subsection*{Case 1.1:}
	Suppose that $|y_0| < R$. We must then have $|y_0| \leq \frac{1}{p}R$. If we fix $|y| = \frac{1}{p}R$ and $|x-x_0| \leq \frac{1}{p}R$, then $D(x_0,|y_0|) \subseteq D(x,|y|)$. We get a nonnegative lower bound on $f(x,y)$ as follows 
	\begin{eqnarray*}
		f(x,y) &=& \frac{1}{N} \sum_{n=1}^N \Xcal_{D(x,|y|) }(\alpha_n) - |y| \cr
		&\geq&  \frac{1}{N} \sum_{n=1}^N \Xcal_{D(x_0,|y_0|) }(\alpha_n) - |y| \cr
		&=& |y_0| + f(x_0,y_0) - |y| \cr
		&=&  |y_0| + \Delta - |y| \cr
		&\geq& \left(1 - \frac{1}{p} \right)R. 
	\end{eqnarray*}
	
	We can bound the $L^2$ norm of $f$ from below by evaluating the required integral only on the set $|y| = \frac{1}{p} R, |x-x_0| \leq \frac{1}{p}R $ 
	\begin{eqnarray*}
		\|f\|_2^2  &=&  \iint_{\ZZ_p^2} |f(x,y)|^2 \, d\mu(x)d\mu(y) \cr
		&\geq&   \iint_{|y|=\frac{1}{p}R, |x-x_0| \leq \frac{1}{p}R} |f(x,y)|^2 \, d\mu(x)d\mu(y) \cr
		&\geq &  \iint_{|y|=\frac{1}{p}R, |x-x_0| \leq \frac{1}{p}R} \left( 1 - \frac{1}{p}\right)^2 R^2 \, d\mu(x)d\mu(y) \cr
		&=&  \left( 1 - \frac{1}{p}\right)^3 \frac{1}{p^2} R^4 \cr
		&\geq& \left( 1 - \frac{1}{p}\right)^3  \frac{1}{p^6}\Delta^4 \cr
		&=& \frac{(p-1)^3 }{p^9} \Delta^4,
	\end{eqnarray*}
	using $R = \lfloor \Delta +|y_0| \rfloor \geq \frac{1}{p}(|y_0| + \Delta) \geq \frac{1}{p} \Delta$. 
	
	\subsection*{Case 1.2: } Suppose that $|y_0|=R$. 
	If we let $|y| = R$ and $|x-x_0| \leq R$, then $D(x_0,|y_0|) = D(x,|y|)$. From this, we get
	\begin{eqnarray*}
		f(x,y) &=& \frac{1}{N} \sum_{n=1}^N \Xcal_{D(x,|y|) }(\alpha_n) - |y| \cr
		&=&  \frac{1}{N} \sum_{n=1}^N \Xcal_{D(x_0,|y_0|) }(\alpha_n) - |y_0| \cr
		&=& |y_0| + f(x_0,y_0) - |y| \cr
		&=&  f(x_0,y_0) \cr
		&=& \Delta. 
	\end{eqnarray*}
	Therefore, 
	\begin{eqnarray*}
		\|f\|_2^2  &=&  \iint_{\ZZ_p^2} |f(x,y)|^2 \, d\mu(x)d\mu(y) \cr
		&\geq&   \iint_{|y|= R, |x-x_0| \leq R} \Delta^2 \, d\mu(x)d\mu(y)  \cr
		&=& \left(1- \frac{1}{p} \right) R^2 \Delta^2 \cr  
		&\geq& \left(1- \frac{1}{p} \right) \frac{1}{(p-1)^2} \Delta^4 \cr
		&=& \frac{1}{p(p-1)} \Delta^4,
	\end{eqnarray*}
	using  $R + \Delta = |y_0| + \Delta < p\lfloor |y_0| + \Delta \rfloor = pR$ and therefore $\Delta < (p-1)R$.
	
	Finally, since $\frac{(p-1)^3}{p^9} < \frac{1}{p(p-1)}$ we conclude 
	\[   \| f\|^2 \geq \frac{(p-1)^3}{p^9} \Delta^4 
	\] 
	holds in both cases 1.1 and 1.2, so it holds in general for case 1. 
	\subsection*{Case 2:} Suppose that $f(x_0,y_0) <  0$ and 
	$ \displaystyle \Delta = |f(x_0,y_0)| = -f(x_0,y_0). $
	In other words, the disc $D(x_0,|y_0|)$ contains fewer than the expected number of points $\alpha_n$. 
	
	Now let $R = |y_0|$. Then if $|y| = R$ and $|x-x_0| \leq R$, by the strong triangle inequality $D(x,|y|) = D(x_0,|y_0|) $ and we have 
	\begin{eqnarray*}
		f(x,y) &=& \frac{1}{N} \sum_{n=1}^N \Xcal_{D(x,|y|) }(\alpha_n) - |y| \cr
		&=&  \frac{1}{N} \sum_{n=1}^N \Xcal_{D(x_0,|y_0|) }(\alpha_n) - |y_0| \cr
		&=& f(x_0,y_0).
	\end{eqnarray*}
	Therefore, 
	\begin{eqnarray*}
		\|f\|_2^2  &=& \iint_{\ZZ_p^2} |f(x,y)|^2 \, d\mu(x)d\mu(y) \cr
		&\geq& \iint_{|y|=R, |x-x_0| \leq R} |f(x,y)|^2 \, d\mu(x)d\mu(y) \cr
		&=& \iint_{|y|=R, |x-x_0| \leq R} \Delta^2 \, d\mu(x)d\mu(y) \cr
		&=& \left(1- \frac{1}{p}\right) R^2 \Delta^2 \cr
		&\geq&  \left(1- \frac{1}{p}\right)\Delta^4,
	\end{eqnarray*}
	where the last line follows because $\Delta \leq R$. To see this, note that 
	\begin{eqnarray*}
		\Delta  &=&  - f(x_0,y_0) \cr
		&=& |y_0| -  \frac{1}{N} \sum_{n=1}^N \Xcal_{D(x_0,|y_0|) }(\alpha_n) \cr
		&\leq& |y_0| \cr
		&=& R.  
	\end{eqnarray*}
	
\end{proof}
Next, we need to prove Lemma \ref{lemma: bound2}. Our goal is to find an upper bound on the $L^2$-norm of $f(x,y)$ using Parseval's theorem. Suppose $f(x,y)$ has a Fourier series 
\[
f(x,y) = \sum_{\zeta, \omega \, \in \ZZ(p^\infty) } \hat{f} (\zeta,\omega) \zeta^x \, \omega^y.
\]  
Then by Parseval's theorem we would get 
\[  \|f \|_2^2 =   \sum_{\zeta, \omega \, \in  \ZZ(p^\infty)  } |\hat{f}(\zeta,\omega)|^2.
\]
Therefore, we need to bound the Fourier coefficients of $f(x,y)$. 
The Fourier coefficients are 
\begin{eqnarray}
\label{eqn: FourierCoefficients}
\hat{f}(\zeta,\omega) & = & \iint_{\ZZ_p^2} f(x,y) \, \zeta^{-x}\omega^{-y} \,\, d\mu(x) d\mu(y) \cr
& = & \frac{1}{N} \sum_{n=1}^N \iint_{\ZZ_p^2} \Xcal_{D(\alpha_n, |y|)}(x) \, \zeta^{-x}\omega^{-y} \,\, d\mu(x) d\mu(y) \,\,  \cr
& &  -  \, \iint_{\ZZ_p^2} |y| \, \zeta^{-x}\omega^{-y} \,\, d\mu(x) d\mu(y).
\end{eqnarray}
Note that if $\zeta = 1$ we get 
\begin{eqnarray}
\label{eqn: zerocoefficients} 
\hat{f}(1,\omega) &=& \frac{1}{N} \sum_{n=1}^N \iint_{\ZZ_p^2} \Xcal_{D(\alpha_n, |y|)}(x) \, \omega^{-y} \,\, d\mu(x) d\mu(y) \,\, -  \, \int_{\ZZ_p} |y| \,\omega^{-y} \,d\mu(y) \cr
&=& \frac{1}{N} \sum_{n=1}^N \int_{\ZZ_p} |y| \omega^{-y} \, d\mu(y) \, - \int_{\ZZ_p} |y| \, \omega^{-y} \,d\mu(y) \cr 
&=& 0. 
\end{eqnarray}
When $\zeta \neq 1$, the second integral in line 2 of Equation  (\ref{eqn: FourierCoefficients}) is zero
\begin{eqnarray*}
	\iint_{\ZZ_p^2} |y| \, \zeta^{-x}\omega^{-y} \,\, d\mu(x) d\mu(y) &=& \int_{\ZZ_p} |y| \omega^{-y} \left(\int_{\ZZ_p} \zeta^{-x} \, d\mu(x) \right) \, d\mu(y) \cr
	&=& 0.
\end{eqnarray*}
Therefore, 
\[
\hat{f}(\zeta, \omega) = \frac{1}{N} \sum_{n=1}^N \iint_{\ZZ_p^2} \Xcal_{D(\alpha_n, |y|)}(x) \, \zeta^{-x}\omega^{-y}\,\, d\mu(x) d\mu(y).
\]
Using Lemma \ref{lemma: charfun}, we have 
\[ 
\int_{\ZZ_p} \Xcal_{D(\alpha_n,|y|)}(x) \zeta^{-x} \, d\mu(x)  = \left\{ 
\begin{array}{cc}
\zeta^{-\alpha_n}\,|y|  &  \text{if \,} \|\zeta \| \leq 1/|y|, \\
& \\
0  &    \text{else}.
\end{array} 
\right.
\]
Hence, for $\zeta \neq 1$, 
\[
\hat{f}(\zeta, \omega) = \frac{1}{N} \sum_{n=1}^N \zeta^{-\alpha_n} \, \int_{|y| \leq 1/\|\zeta\| } |y| \, \omega^{-y} \, d\mu(y).
\]
The following lemma makes some estimates that are useful in our succeeding calculations
\begin{lem}
	\label{lemma: IntegralEst}
	Let $R = p^k, k \in \ZZ$ satisfy $0 < R < 1$, and let $\omega \in \ZZ(p^\infty).$ Then,
	\begin{equation}
	\label{eqn: IntegralEst1}
	\left|\int_{|y| \leq R} |y| \, \omega^{-y} \, d\mu(y)  \right| \leq \frac{p}{\max(1/R, \|\omega \|)^2}. 
	\end{equation} 	
	Moreover, 
	\begin{equation}
	\label{eqn: IntegralEst2}
	\sum_{\omega \in \ZZ(p^\infty)} \left|\int_{|y| \leq R} |y| \, \omega^{-y} \, d\mu(y)  \right|^2 \leq 2p^2R^3. 
	\end{equation} 	
\end{lem}

\begin{proof}[Proof of Lemma \ref{lemma: IntegralEst}] 
	Let $R = 1/p^k$ and let $\|\omega\| = p^l$. We have 
	\begin{eqnarray}
	\label{eqn: IntegralEst3}
	\int_{|y| \leq R} |y| \, \omega^{-y} \, d\mu(y) &=& \sum_{j \geq k} \frac{1}{p^j} \int_{|y| = 1/p^j } \omega^{-y} \, d\mu(y) \cr
	&=& \sum_{j \geq k} \frac{1}{p^j} \int_{\ZZ_p} \left(\Xcal_{D(0,1/p^j) }(y) - \Xcal_{D(0,1/p^{j+1}) }(y)\right)
	\omega^{-y} \, d\mu(y).
	\end{eqnarray}
	When $\|\omega\| \leq 1/R$, that is when $l \leq k$, using Lemma \ref{lemma: charfun} and (\ref{eqn: IntegralEst3}) we have 
	\begin{eqnarray*} 
		\int_{|y| \leq R} |y| \, \omega^{-y} \, d\mu(y) &=& \sum_{j \geq k} \frac{1}{p^j}\left(\frac{1}{p^j} - \frac{1}{p^{j+1}} \right) \cr
		&=& \frac{p}{(p+1)(1/R)^2}. 
	\end{eqnarray*}
	Thus (\ref{eqn: IntegralEst1}) holds in this case. If $\|\omega\| > 1/R$, that is when $l \geq k+1$, again using Lemma \ref{lemma: charfun} and (\ref{eqn: IntegralEst3}) we have 
	\begin{eqnarray*} 	
		\int_{|y| \leq R} |y| \, \omega^{-y} \, d\mu(y) &=& \sum_{j \geq l} \frac{1}{p^j}\left(\frac{1}{p^j} - \frac{1}{p^{j+1}} \right) - \frac{1}{p^{l-1}p^l} \cr
		&=& -\frac{p^2}{(p+1)\|\omega\|^2},
	\end{eqnarray*}
	and thus (\ref{eqn: IntegralEst1}) holds in this case as well.
	
	To check (\ref{eqn: IntegralEst2}) , we use the fact that for each $j \geq 1$, the group $\ZZ(p^\infty)$ contains $p^j$ elements of order at most $p^j$, and $p^j - p^{j-1}$ elements of order exactly $p^j$. We then have 
	\begin{eqnarray*}
		\sum_{\omega \in \ZZ(p^\infty)} \left|\int_{|y| \leq R} |y| \, \omega^{-y} \, d\mu(y)  \right|^2 &\leq& p^2 \sum_{\omega \in \ZZ(p^\infty)} \frac{1}{\max(1/R, \|\omega \|)^4} \cr
		&=& p^2\left(\sum_{\|\omega\| \leq 1/R} R^4 + \sum_{\|\omega\| > 1/R}  \frac{1}{\|\omega\|^4}\right) \cr
		&=& p^2  \left(R^3 + \frac{p-1}{p(p^3-1)}R^3 \right) \cr
		&<& 2p^2R^3.
	\end{eqnarray*}
\end{proof} 
Finally, we prove Lemma \ref{lemma: bound2}.
\begin{proof}[Proof of Lemma \ref{lemma: bound2}]
	Applying Parseval's theorem to $f(x,y)$ and using Equations (\ref{eqn: zerocoefficients}) and (\ref{eqn: IntegralEst2}), we conclude
	\begin{eqnarray*}
		\|f\|^2 &=& \sum_{\zeta, \omega \in \ZZ(p^\infty) \atop \zeta \neq 1} |\hat{f}(\zeta,\omega)|^2 \cr
		&=& \sum_{\zeta \in \ZZ(p^\infty) \atop \zeta \neq 1}
		\left( \sum_{\omega \in \ZZ(p^\infty)} 
		\left| \int_{|y| \leq 1/\|\zeta\|} |y|\omega^{-y} \, d\mu(y) \right|^2 \right) \left|\frac{1}{N} \sum_{n= 1}^N \zeta^{\alpha_n} \right|^2 \cr
		&\leq& 2p^2 \sum_{\zeta \in \ZZ(p^\infty) \atop \zeta \neq 1}
		\frac{1}{\| \zeta\|^3} \left|\frac{1}{N} \sum_{n= 1}^N \zeta^{\alpha_n} \right|^2.
	\end{eqnarray*}
	
\end{proof}


\section{The linear sequence $na+b$ in $\ZZ_p$}
\label{section: Linear-sequence}
Consider the sequence $\alpha_n = na + b $. We have the following proposition, a proof of which is given in \cite{book:MR0419394} using elementary number theory. We present an alternate proof using Fourier analysis.
\begin{prop}
	\label{prop: linear-sequence}
	The sequence $\alpha_n = n a + b$ is equidistributed in $\ZZ_p$ if and only if $a$ is a unit in $\ZZ_p$.
\end{prop}
\begin{proof} 
	The forward implication follows from Weyls criterion (Proposition \ref{prop: Weyl}). For suppose, $a$ was not a unit. Then $a = p^kc$, where $k>0$ and $c$ is a unit. Now let $\zeta = e^{2\pi i/p^k}$.  Then $\zeta^a =1$, and Weyl's criterion will not hold. 
	
	For the reverse implication, let $\zeta \in \ZZ(p^\infty)$ with $\| \zeta \| = p^k$ for $k \geq 1$. There exists an $m$ such that $1 \leq m < p^k $, with $p \ndiv m$ and $\displaystyle \zeta = e^{2\pi im/p^k}$. Suppose that $a $ is a unit in $\ZZ_p$. Let $a = t_0 + t_1p + t_2p^2 +...$ be the canonical expansion of $a$, with $t_0 \neq 0$.  Then we let $a_k = t_0 + t_1p + ... + t_{k-1}p^{k-1}$ be the truncation of this expansion to the first $k$ terms. We have 
	\begin{eqnarray}
	\frac{1}{N}\left| \sum_{n=1}^{N} \zeta^{na + b} \right|
	&=&  \frac{1}{N}\left| \sum_{n=1}^{N} \zeta^{na} \right| \cr
	&  & \cr
	&=&    \frac{1}{N} \left| \frac{ 1-   \zeta^{(N+1)a_k} }{1 -  \zeta^{a_k}}  \right| \cr
	& & \cr   \label{eqn: lin-sequence}
	&\leq& \frac{1}{N} \frac{2}{\left|1 -  \zeta^{a_k}  \right| } \cr
	&\leq& \frac{1}{N} \left| \frac{1}{\sin(\pi m a_k/p^k)} \right|.
	\end{eqnarray}
Since $p \ndiv m$ and $p \ndiv a_k$, $\displaystyle \sin(\pi m a_k/p^k) \neq 0$ and hence $\frac{1}{N} \sum_{n=1}^{N} \zeta^{na + b}  \to 0$ as $N \to \infty$; the proof of equidistribution now follows from Weyl's criterion. 
\end{proof}

\begin{proof}[Proof of Corollary~\ref{cor: Linear-sequence}]
	Applying the bound given by Theorem \ref{thm: Main}  we get
	\begin{eqnarray}
	\label{eqn: CorEqn1}
	D_N^4 &\ll& \sum_{\zeta \in \ZZ(p^\infty)\backslash \{1\}  }   \frac{1}{\|\zeta \|^3} \left|\frac{1}{N} \sum_{n=1}^{N} \zeta^{  na + b} \right|^2 \cr
	& & \cr
	&\leq& \frac{1}{N^2} \sum_{k=1}^\infty \frac{1}{p^{3k}}  \sum_{1 \leq m < p^k \atop p \nmid m} \frac{1}{ \left|\sin \left(\pi \, m a_k/ p^k \right) \right|^2} \cr
	& & \cr
	&\leq& \frac{1}{N^2} \sum_{k=1}^\infty \frac{1}{p^{3k}}  \sum_{1 \leq m < p^k} \frac{1}{ \left|\sin \left(\pi \, m a_k/ p^k \right) \right|^2} \cr
	& & \cr
	&\leq& \frac{1}{N^2} \sum_{k=1}^\infty \frac{1}{p^{3k}}  \sum_{1 \leq l < p^k} \frac{1}{ \left|\sin \left(\pi \, l/ p^k \right) \right|^2} \cr
	& & \cr
	&\leq& \frac{2}{N^2} \sum_{k=1}^\infty \frac{1}{p^{3k}}  \sum_{1 \leq l \leq p^k/2} \frac{1}{ \left|\sin \left(\pi \, l/ p^k \right) \right|^2}.
	\end{eqnarray}
	
	Note that the second inequality in Equation (\ref{eqn: CorEqn1}) comes from the last inequality in (\ref{eqn: lin-sequence}). For the fourth inequality, note that since $a$ is a unit we have $p \ndiv a_k$. Hence, $\gcd(a_k,p^k) =1$ and so $a_k$ generates $\ZZ/p^k\ZZ$. That is, $\ZZ/p^k\ZZ = \left\{ma_k \, | \, m = 0,..,p^k-1 \right\}$. The final inequality follows from the identities $|\sin(\theta)| = |\sin(-\theta)| = |\sin(\pi - \theta)|$, so that for $p^k/2 \leq l <p^k$ we have $|\sin(\pi l /p^k)| = |\sin(\pi(p^k-l)/p^k) |$. This allows us to double the sum over the first half of the interval. 
	
	Note that in the interval $[0,\, \pi/2]$, $\sin(\theta)$ is bounded from below by $2\theta/\pi$, so that 
	\[  \frac{1}{|\sin(\theta)|} \leq \frac{\pi}{2\theta}. 
	\]
	This gives us 
	\begin{eqnarray} 
	\label{eqn: CorEqn2}
	\sum_{1 \leq l \leq p^k/2} \frac{1}{ \left|\sin \left(\pi \, l/ p^k \right) \right|^2} 
	&\leq&  \sum_{1 \leq l \leq p^k/2} \frac{p^{2k} }{4 l^2 } \cr
	& & \cr
	&\leq& \frac{p^{2k}}{4} \sum_{1 \leq l <\infty} \frac{1}{l^2} \cr
	&  & \cr
	&\leq& \frac{p^{2k} \pi^2}{24}. 
	\end{eqnarray} 
	
	Finally, applying the bound from (\ref{eqn: CorEqn2}) to (\ref{eqn: CorEqn1}) we get
	\begin{eqnarray*}
		D_N^4  &\ll& \frac{\pi^2}{12N^2} \sum_{k=1}^{\infty} \frac{1}{p^{k}}.
	\end{eqnarray*}
	We conclude that $D_N =  O\left(\frac{1}{\sqrt{N}}\right)$.
\end{proof}


\bibliographystyle{siam}
\bibliography{Equidistribution-Zp-paper}

\end{document}